\documentclass[a4paper,12pt]{amsart}
\usepackage{lmodern}
\usepackage{amsmath,amsfonts,amssymb,mathrsfs,charter,amsthm, cleveref, enumitem}

\newcommand{\bC}{\mathbb{C}}

\newcommand{\dist}{\mathop{{\rm dist}}}
\newcommand{\var}{\mathop{{\rm var}}}
\newcommand{\Hil}{\mathscr{H}}

\newcommand{\h}{\mathop{{\rm Re}}}
\newcommand{\hl}{\mathop{{\rm Im}}}
\newcommand{\tr}{{\mathrm{tr }}}
\newcommand{\conv}{\mathop{{\rm conv}}}
\newcommand{\bA}{\mathop{{ \boldsymbol A}}}

\newcommand{\I}{\mathop{{ \boldsymbol I}}}
\newcommand{\z}{\mathop{{ \boldsymbol z}}}

\newcommand{\0}{\mathop{{ \boldsymbol 0}}}
\newtheorem{thm}{Theorem}

\newtheorem{corollary}[thm]{Corollary}

\newtheorem{prop}[thm]{Proposition}

\newtheorem{lemma}[thm]{Lemma}
\newcommand{\bra}{\langle}
\newcommand{\ket}{\rangle}
\newcommand{\tens}[1]{\mathbin{\mathop{\otimes}\limits_{#1}}}

\title{A distance formula  for tuples of operators}
\author[Grover]{Priyanka Grover}
\author[Singla]{Sushil Singla}
\address{Department of Mathematics, Shiv Nadar University, NH-91, Tehsil Dadri, Gautam Buddha Nagar, U.P. 201314, India.}
\email{priyanka.grover@snu.edu.in, ss774@snu.edu.in}

\begin{document}

\maketitle

\begin{abstract}
For a tuple of operators $\boldsymbol{A}= (A_1, \ldots, A_d)$, $\dist (\boldsymbol{A}, \mathbb C^d \boldsymbol{I})$ is defined as $\min\limits_{\boldsymbol{z} \in \mathbb C^d} \|\boldsymbol{A-zI}\|$ 
and 
$\var\limits_x (\boldsymbol{A})$ as $\|\boldsymbol{A}  x\|^2-\sum_{j=1}^d {\big|}\langle x| A_j x\rangle{\big|}^2.$ For a tuple $\boldsymbol{A}$ of commuting normal operators, it is known that $$\dist  (\boldsymbol{A}, \mathbb C^d \boldsymbol{I})^2=\sup_{\|x\|=1}\var_x (\boldsymbol{A}).$$ 
We give an expression for the maximal joint numerical range of a tuple of doubly commuting matrices. Consequently, we obtain that the above distance formula holds for tuples of doubly commuting matrices. We also discuss some general conditions on the tuples of operators for this formula to hold. As a result, we obtain that it holds for tuples of Toeplitz operators as well.
\end{abstract}

\setlength{\parindent}{0pt}
\setlength{\parskip}{1.6ex}

\pagestyle{headings}

\section{Introduction}
Let $\Hil$ be a Hilbert space. Let $\mathscr B(\Hil)$ be the space of bounded linear operators on $\Hil$. Bj\"{o}rck and Thom\'ee \cite{bjorckthomee} showed that for a bounded normal operator $A$ on $\Hil$,
\begin{equation*}
    \sup_{\|x\|=1} \left(\|Ax\|^2-{\big|}\langle x|Ax\rangle{\big|}^2\right)=R_A^2,
\end{equation*}
where $R_A$ is the radius of the smallest disc containing the spectrum of $A$. The quantity $ \|A  x\|^2- {\big|}\langle x|A x\rangle{\big|}^2$ is the \emph{variance} of $A$ with respect to $x$, denoted by $\var\limits_x (A)$ (see \cite{audenaert}). Later,
Garske \cite{Garske} proved that a one side inequality is true for any operator $A$ on $\Hil$: \begin{equation}\sup_{\|x\|=1} \var\limits_x (A)\geq R_A^2.\label{eq1}\end{equation} Let $\Hil^d$ be the direct sum of $d$ copies of $\Hil$. For $A_1,\ldots, A_d\in \Hil$, let $\bA$ denote the tuple $(A_1, \ldots, A_d):\Hil\rightarrow \Hil^d$ defined as $\bA x=(A_1 x,\ldots A_d x)$. Then
$$\|\bA\|=\left\|\sum_{j=1}^d A_j^* A_j\right\|^{1/2}.$$ We define the variance of the tuple $\bA$ with respect to $x\in \Hil$ as $\var\limits_x (\bA)= \|\bA x\|^2-\sum\limits_{j=1}^d{\big|}\langle x|A_jx\rangle{\big|}^2$. Ming \cite{Fan} showed that if $\bA$ is a tuple of commuting operators on $\Hil$, then \begin{equation}
    \sup_{\|x\|=1} \var\limits_x (\bA) \geq R_{\boldsymbol{A}}^2,\label{inequality}
\end{equation}
where $R_{\bA}$ is the radius of the smallest disc containing the Taylor spectrum of $\bA$. Moreover, if each $A_j$ is normal, \eqref{inequality} becomes an equality. 
(See also \cite{bhatiasharma}.)

Let $\I$ be the tuple of identity operators $(I, \dots, I)$. For $\z=(z_1,\ldots,z_d)\in \mathbb C^d$, $\z\I$ denotes the tuple $(z_1 I,\ldots,z_d I)$. Let ${\dist(\bA, \mathbb C^d \I)=\min\limits_{\z \in \bC^d} \|\bA-\z \I\|}.$  
  For commuting normal operators, $R_{\bA} =\dist(\bA, \mathbb C^d \I),$ and thus we have
\begin{equation}
\sup_{\|x\|=1}\var_x (\bA)=\dist  (\bA, \mathbb C^d \I)^2. \label{main2}
\end{equation}

We note that for any tuple of operators $\bA$, we have
\begin{equation}
\sup_{\|x\|=1}\var_x (\bA)\leq\dist  (\bA, \mathbb C^d \I)^2. \label{main3}
\end{equation}

To see this, let $x$ be a unit vector in $\Hil$. Then for $\z\in \mathbb C^d$ and $1\leq j\leq d$, we have\begin{align*}
\|A_j x\|^2-\big|\langle x| A_j x\rangle\big|^2 &= \|(A_j-z_j I) x\|^2-\big|\langle x|(A_j-z_j I) x\rangle\big|^2\\ 
&\leq \|(A_j-z_jI)x\|^2.
\end{align*}
Thus for all $\z\in \mathbb C^d$,
\begin{equation*}
\sum_{j=1}^d \big(\|A_j x\|^2-\big|\langle x|A_jx\rangle\big|^2\big) \leq \sum_{j=1}^d \|(A_j-z_j I)x\|^2\leq \|\bA-\z \I\|^2.\end{equation*}
So $\var\limits_x (\bA)\leq \dist  (\bA, \mathbb C^d \I)^2$. Since $x$ was an arbitrary unit vector in $\Hil$, we get \eqref{main3}. 

A \emph{doubly commuting tuple} of operators is one which satisfies $A_i A_j=A_j A_i $ and $A_i^* A_j=A_j A_i^*$ for $1\leq i,j\leq d, i\neq j$. In particular, a tuple of commuting normal operators is doubly commuting. We show that for a finite dimensional Hilbert space $\Hil$, \eqref{main2} holds for a doubly commuting tuple of operators.

We also give some equivalent conditions for any tuple of operators on any Hilbert space to satisfy \eqref{main2}. For a tuple of operators $\bA$ on a Hilbert space $\Hil$, the \emph{joint maximal numerical range} of $\bA$ is defined as \begin{align*}W_0(\bA) = & \{(\lambda_1,\dots,\lambda_d)\in\bC^d :\langle x^{(n)}|A_j x^{(n)}\rangle\rightarrow \lambda_j \text{ for all } 1\leq j\leq d,\\
&\text{where }\|x^{(n)}\|=1\text{ and }\|\bA x^{(n)}\|\rightarrow\|\bA\|\}.
\end{align*} 
We shall denote by $\boldsymbol{z^0}=(z_1^0,\dots, z_d^0)\in\bC^d$, the unique element for which $\dist(\boldsymbol{A}, \bC^d \boldsymbol{I}) = \|\boldsymbol{A- z^0I}\|$. Let $\boldsymbol{A^0}$ denote $\boldsymbol{A- z^0I}$ and for each $1\leq j\leq d$, let $A_j^0=A_j-z_j^0 I$. It is shown that the convexity of $W_0(\boldsymbol{A^0})$ is sufficient for $\bA$ to satisfy \eqref{main2}. As a corollary, we obtain that \eqref{main2} holds for tuples of Toeplitz operators as well. We give an example to show that the convexity of $W_0(\boldsymbol{A^0})$ is not necessary for $\bA$ to satisfy \eqref{main2}. Convexity of the \emph{joint numerical range} and the joint maximal numerical range has been a subject of interest for many authors (see \cite{jnr2, jnr3, jnr8, dashtensor, jnr6, jnr1, jnr4, jnr7, jnr5}).

In Section \ref{section1}, we provide results for doubly commuting matrices. In Section \ref{section2}, we give some conditions for \eqref{main2} to hold for general tuples of operators and obtain results for Toeplitz operators. In Section \ref{section3}, we give some remarks.

\section{Doubly commuting matrices}\label{section1}
For $x_1, x_2\in\mathbb C^n$, let $x_1\bar{\tens{}} x_2$ denote the rank one operator on $\mathbb C^n$ defined as $(x_1\bar{\tens{}} x_2)(y) = \bra x_2 | y\ket x_1$ for all $y\in\mathbb C^n$. Let $\boldsymbol{0}$ be the tuple $(0,\dots, 0)\in\mathbb C^d$.
\begin{thm}\label{orthogonality1} For any tuple of matrices $\bA$, $$\|\bA-\z \I\|\geq \|\bA\| \text{ for all } \z\in \mathbb C^d$$ if and only if 
$$\boldsymbol{0}\in\conv\left\{\left(\langle x|A_1 x \rangle, \dots , \langle x|A_d x\rangle\right): \|x\|=1, \textstyle\bA^* \bA x=\|\bA\|^2 x \right\}.$$
\end{thm}
\begin{proof} Using Theorem 8.4 of \cite{Zietak2}, $\|\bA-\z \I\|\geq \|\bA\|$ for all $\z\in \mathbb C^d$ if and only if there exists a positive semidefinite matrix $T$ with $\tr(T) = 1$ such that $\bA^*\bA T=\|\bA\|^2 T$ and $\tr(\bA^*\z\I T)=0$ for all $\z\in\mathbb C^d$. Using the spectral theorem for $T$, there are positive numbers $s_1, s_2, \dots, s_\ell$ such that $\sum\limits_{i=1}^\ell s_i =1$ and orthonormal vectors $x_1, \dots, x_\ell$ such that $T=\sum\limits_{i=1}^\ell s_i x_i\overline{\tens{}}x_i$. We note that $\bA^*\bA T=\|\bA\|^2 T$ is equivalent to $\bA^* \bA x_i = \|\bA\|^2x_i$ for each $1\leq i\leq \ell$. To see this, we observe that \begin{align*}
\|\bA\|^2 &= \tr(\boldsymbol{A}^*\bA T) \\
&= \sum\limits_{i=1}^\ell s_i \bra x_i|\boldsymbol{A}^*\bA x_i\ket\\
&\leq \|\boldsymbol{A}^*\bA\|. \end{align*}

So $$\sum\limits_{i=1}^\ell s_i\bra x_i|\boldsymbol{A}^*\bA x_i\ket = \|\boldsymbol{A}^*\bA\|= \sum\limits_{i=1}^\ell s_i\|\boldsymbol{A}^*\bA\|.$$ By the condition of equality in the Cauchy-Schwarz inequality, we get $\boldsymbol{A}^*\bA x_i=\|\bA\|^2x_i$ for all $1\leq i\leq\ell$.
We also have that $\tr(\bA^*\z\I T)=0$ for all $\z\in\mathbb C^d$ is equivalent to $\sum\limits_{i=1}^\ell s_i\bra x_i |A_j x_i\ket = 0$ for all $1\leq j\leq d$. This gives the required result.
\end{proof}


Let $W(\bA)=\{(\langle x|A_1 x \rangle, \langle x|A_2 x\rangle, \ldots, \langle x|A_d x\rangle): \|x\|=1\}$ denote the joint numerical range of $\bA$.
Let  $$\mathcal V(\bA)=\{\left(\langle x|A_1 x \rangle, \langle x|A_2 x\rangle, \ldots, \langle x|A_d x\rangle\right): \|x\|=1, \boldsymbol{A}^* \bA x=\|\bA\|^2 x \}.$$ Note that $\mathcal V(\bA)=W(P\bA P)$, where $P$ is the orthogonal projection of $\Hil$ onto the subspace $\{x:  \textstyle\bA^* \bA x=\|\bA\|^2 x\}$. As a consequence of Theorem \ref{orthogonality1}, we have 

\begin{corollary}\label{orthogonality} Let $\boldsymbol{A}$ be a tuple of matrices such that $\mathcal V(\bA)$ is a convex set. Then we have $$\|\boldsymbol{A}-\boldsymbol{z} \boldsymbol{I}\|\geq \|\boldsymbol{A}\|  \text{ for all } \boldsymbol{z}\in \mathbb C^d$$ if and only if there exists a unit vector $x\in\mathbb C^n$ such that $\|\boldsymbol{A}x\|=\|\boldsymbol{A}\|$ and $\langle x|A_jx \rangle=0 \text{ for all } 1\leq j\leq d$.
\end{corollary}

For a tuple of matrices, $\mathcal V(\bA)=W_0(\bA)$. It was proved in \cite[Theorem 1]{jnr9} that if $\bA$ is a tuple of operators such that $W(\bA)$ is convex, then so is $W_0(\bA)$. Bolotnikov and Rodman  \cite[Theorem 3.1]{rodman} showed that if $\bA$ is a tuple of doubly commuting matrices, then $W(\bA)$ is convex.
Thus we have the following lemma.

\begin{lemma}\label{doubly}
For a doubly commuting tuple of matrices $\bA$, $\mathcal V(\bA)$ is convex.
\end{lemma}

In Theorem 1.1 of \cite{rodman}, the authors gave a complete description of a doubly commuting tuple $\bA$ of matrices. They showed that there is a unitary matrix $U$, positive integers $m_1,\ldots,m_\ell$, $m_k\times m_k$ matrices $A_{1,k}, \ldots, A_{d,k}$ for each $1\leq k \leq \ell$ such that  $n=m_1+m_2+\dots + m_\ell$ and for $j=1,\ldots,d$,
\begin{equation}\label{chap5:55}UA_j U^*=\begin{bmatrix}A_{j,1}& 0& \ldots&0\\ 0& A_{j,2}& \ddots&0\\ \vdots&\vdots&\ddots&\vdots\\ 0&\ldots&0& A_{j,\ell}\end{bmatrix}.\end{equation} Moreover, for $1\leq k\leq \ell$,  $m_k=p_{1,k}p_{2,k}\dots p_{d,k}$ and \begin{eqnarray*} A_{1,k} &=& X_{1,k}\otimes I_{p_{2, k}}\otimes\dots\otimes I_{p_{d,k}},\nonumber\\
A_{2,k} &=& I_{p_{1,k}}\otimes X_{2,k}\otimes I_{p_{3, k}}\otimes\dots\otimes I_{p_{d,k}},\nonumber\\
& &\vdots\nonumber\\
A_{d,k} &=& I_{p_{1,k}}\otimes I_{p_{2, k}}\otimes\dots\otimes I_{p_{d-1,k}}\otimes X_{d,k}, \label{mkpjk}
\end{eqnarray*}
where for $1\leq j\leq d$, $X_{j,k}$ are $p_{j,k}\times p_{j,k}$ matrices and $I_{p_{j,k}}$ denote the $p_{j,k}\times p_{j,k}$ identity matrices. Using this, they proved that for doubly commuting tuples of matrices $\bA$, $W(\bA)$ is convex. Modifying their proof along with \cite[Proposition 4]{dashtensor} (which is also the main ingredient in \cite[Theorem 3.1]{rodman}), we show that
\begin{equation}\mathcal V(\boldsymbol{A})=\conv\left\{\prod\limits_{j=1}^d\mathcal V(X_{j, k}) :  1\leq k\leq \ell, \sum_{j=1}^d\|X_{j,k}\|^2=\|\boldsymbol{A}\|^2\right\}.\label{equation6}\end{equation} The convexity of $\mathcal V(\bA)$ (Lemma \ref{doubly}) follows as a consequence. We prove \eqref{equation6} in two steps given in the subsequent lemmas.

\begin{lemma}\label{lemma2}

Let \begin{eqnarray*}
A_1&=&X_1\otimes I_{p_2} \otimes \cdots\otimes I_{p_d},\\
& & \vdots\\
A_d&=&I_{p_1}\otimes I_{p_2} \otimes \cdots\otimes X_d,\end{eqnarray*} where each $X_j$ is a $p_j\times p_j$ matrix and $I_{p_j}$ is $p_j\times p_j$ identity matrix. Then
$\mathcal V(A_{1},\dots, A_{d})=\prod\limits_{j=1}^d\mathcal V(X_{j})$.

\end{lemma}

\begin{proof} We have $$\sum\limits_{j=1}^d A_{j}^*A_{j} = \sum\limits_{j=1}^d I_{p_1}\tens{}\dots\tens{}I_{p_{j-1}}\tens{}(X_{p_j}^*X_{j})\tens{}I_{p_{j+1}}\dots\tens{}I_{p_d}.$$ So the eigenvalues of $\sum\limits_{j=1}^d A_{j}^*A_{j}$ are precisely the sums of the eigenvalues of $X_{j}^*X_{j}$ over $j$. This gives
\begin{equation}\left\|\sum\limits_{j=1}^d A_{j}^*A_{j}\right\|=\sum\limits_{j=1}^d\left\|X_{j}^*X_{j}\right\|.\label{new}\end{equation} 

Let $\left(\bra x_1|X_{1}x_{1}\ket, \ldots, \bra x_{d}|X_{d}x_{d}\ket\right)\in \prod\limits_{j=1}^d\mathcal V(X_{j})$. So $$X_{j}^*X_{j}x_j=\|X_{j}\|^2 x_{j}\text{ for all }1\leq j\leq d.$$ Thus \begin{align*}\sum\limits_{j=1}^d A_{j}^*A_{j}\ (x_1\otimes\cdots\otimes x_d)&=\sum\limits_{j=1}^d \|X_{j}\|^2 (x_1\otimes\cdots\otimes x_d)\\
&=\left\|\sum\limits_{j=1}^d A_{j}^*A_{j}\right\| (x_1\otimes\cdots\otimes x_d).\end{align*}
Since $$\bra x_j|X_{j}x_j\ket=\bra x_1\otimes\cdots\otimes x_d|A_{j}(x_1\otimes\cdots\otimes x_d)\ket \text{ for all }1\leq j\leq d,$$  we obtain $\left(\bra x_1|X_{1}x_1\ket, \ldots, \bra x_d|X_{d}x_d\ket\right)\in \mathcal V(A_{1},\dots, A_{d}).$

Now let $\left(\bra x| A_{1}x\ket,\ldots, \bra x|A_{d}x\ket\right)\in\mathcal V(A_{1},\dots, A_{d})$, where $\|x\|=1$ and $$\left(\sum\limits_{j=1}^dA_{j}^*A_{j}\right)x=\left\|\sum\limits_{j=1}^d A_{j}^*A_{j}\right\|x = \sum\limits_{j=1}^d\left\|X_{j}^*X_{j}\right\|x .$$
For $1\leq j\leq d$, let $\{e_{q}^{(j)}:q=1,\ldots,p_{j}\}$ be an orthonormal basis of $\bC^{p_{j}}$.
Fix $i\in\{1,\dots,d\}$ and let \begin{equation*} S_i =\big\{(q_1, \ldots, q_{i-1}, q_{i+1}, \ldots, q_d) : 1\leq q_j\leq p_j \text{ for all }j\in\{1,\dots, d\}\setminus\{i\}\big\}.\end{equation*} 
Then we have \begin{equation*}x=\sum\limits_{(q_1, \dots, q_{i-1}, q_{i+1}, \dots, q_d)\in S_i} e_{q_1}^{(1)}\otimes\dots\otimes e_{q_{i-1}}^{(i-1)}\otimes f_{(q_1, \dots, q_{i-1}, q_{i+1}, \dots, q_d)}
\otimes e_{q_{i+1}}^{(i+1)}\otimes\dots\otimes e_{q_d}^{(d)},\end{equation*} where $f_{(q_1, \dots, q_{i-1}, q_{i+1}, \dots, q_d)}\in \bC^{p_{i}}$ and
$$1=\|x\|^2=\sum_{\substack{\alpha\in S_i}} \|f_{\alpha}\|^2.$$
The condition $\left\|\left(\sum\limits_{j=1}^d A_{j}^*A_{j}\right)x\right\|=\sum\limits_{j=1}^d\|X_{j}^*X_{j}\|$ gives $$X_{i}^*X_{i}f_{\alpha}=\|X_{i}^*X_{i}\|f_{\alpha}\text{ for all }f_{\alpha} \in \bC^{p_{i}}\setminus \{0\}.$$

Since \begin{align*}\bra x|A_{i}x\ket &=\sum\limits_{\substack{\alpha\in S_i}}\bra f_{\alpha}|X_{i}f_{\alpha}\ket\\
& = \sum\limits_{\substack{\alpha\in S_i,\\ f_{\alpha}\neq 0 }} \|f_{\alpha}\|^2\left\bra \dfrac{f_{\alpha}}{\|f_{\alpha}\|}\bigg{|} X_{i}\dfrac{f_{\alpha}}{\|f_{\alpha}\|}\right\ket,\end{align*} it lies in $\mathcal V(X_{i})$. Hence $$\left(\bra x| A_{1}x\ket,\ldots, \bra x|A_{d}x\ket\right)\in\prod_{i=1}^d\mathcal V(X_{i}).$$ This gives $\mathcal V({A}_1,\ldots, A_d)=\prod\limits_{i=1}^d\mathcal V(X_{i})$.

\end{proof}
\begin{lemma}\label{lemma1}
For a doubly commuting tuple of matrices $\bA$, $$\mathcal V(\bA)=\conv\left\{\mathcal V(A_{1,k},\dots, A_{d,k}): 1\leq k\leq \ell, \left\|\sum_{j=1}^d A_{j,k}^*A_{j,k}\right\|=\|\boldsymbol{A}\|^2\right\},$$ where $A_{j,k}$ $(1\leq j\leq d, 1\leq k\leq\ell)$ are as given in \eqref{chap5:55}.
\end{lemma}
\begin{proof}
Let $x$ be a unit vector such that $\bA^* \bA x=\|\bA\|^2 x$. Let $y=Ux=\left(y_1,  \ldots, y_\ell\right).$
Fix $1\leq k\leq \ell$. Then
\begin{equation*}
\left(\sum_{j=1}^d A_{j,k}^*A_{j,k}\right)y_k=\|\bA\|^2 y_k. \label{eq3}
\end{equation*}
Note that $\left\|\sum_{j=1}^d A_{j,k}^*A_{j,k}\right\|\leq \|\boldsymbol{A^*A}\|$. Thus $y_k=0$ whenever $ \left\|\sum_{j=1}^d A_{j,k}^*A_{j,k}\right\|\neq\|\boldsymbol{A}\|^2$. Now for $1\leq k\leq \ell$ such that $y_k\neq 0$, we have \begin{align*}\|\bA\|^2\|y_k\|&=\left\|U\left(\sum_{j=1}^d A_j^* A_j\right)U^*\right\|\|y_k\|\\
&\geq\left\|\sum_{j=1}^d A_{j,k}^*A_{j,k}\right\|\|y_k\|\\&\geq\left\|\left(\sum_{j=1}^d A_{j,k}^*A_{j,k}\right)y_k\right\|.\end{align*} So
\begin{equation*}\left(\sum_{j=1}^d A_{j,k}^*A_{j,k}\right)y_k=
\left\|\sum_{j=1}^d A_{j,k}^*A_{j,k}\right\| y_k.\label{eq5}\end{equation*}
Also, for each $1\leq j\leq d$, $$\langle x|A_j x\rangle=\sum\limits_{\substack{1\leq k\leq\ell,\\ y_k\neq 0}}\|y_k\|^2\left\langle \dfrac{y_k}{\|y_k\|}\bigg| A_{j,k} \dfrac{y_k}{\|y_k\|}\right\rangle .$$ 
Thus \begin{equation*}\mathcal V(\bA)\subseteq\conv\left\{\mathcal V(A_{1,k},\dots, A_{d,k}): 1\leq k\leq \ell, \left\|\sum_{j=1}^d A_{j,k}^*A_{j,k}\right\|=\|\boldsymbol{A}\|^2\right\}.\end{equation*}

Now let $$\boldsymbol{\lambda}\in \conv\left\{\mathcal V(A_{1,k},\dots, A_{d,k}): 1\leq k\leq \ell, \left\|\sum_{j=1}^d A_{j,k}^*A_{j,k}\right\|=\|\boldsymbol{A}\|^2\right\}.$$ We have $$\|\boldsymbol{A}\|^2 = \max\left\{\left\|\sum_{j=1}^d A_{j,1}^*A_{j,1}\right\|, \dots,\left\|\sum_{j=1}^d A_{j,\ell}^*A_{j,\ell}\right\|\right\}.$$ Let $m$ be the number of terms where this maximum is attained. Without loss of generality, let $1\leq k\leq m$ be such that $\left\|\sum_{j=1}^d A_{j,k}^*A_{j,k}\right\|=\|\boldsymbol{A}\|^2$. So there are positive numbers $s_1, \dots, s_m$ such that $\sum\limits_{k=1}^m s_k = 1$ and $\boldsymbol{\lambda}_k\in  \mathcal V(A_{1,k},\dots, A_{d,k})$ for all $1\leq k\leq m$ such that \begin{equation}\label{589}\boldsymbol{\lambda} = \sum\limits_{k=1}^m s_k \boldsymbol{\lambda}_k.\end{equation}

Since $\boldsymbol{\lambda}_k\in  \mathcal V(A_{1,k},\dots, A_{d,k})$, there exists a unit vector $y_k\in \mathbb C^{m_k}$ such that \begin{equation}\label{eq10}\boldsymbol{\lambda}_k = \left(\langle y_k|A_{1, k} y_k\rangle, \dots, \langle y_k|A_{d, k} y_k\rangle\right),\end{equation} and $$\left(\sum\limits_{j=1}^{d} A_{j, k}^*A_{j, k}\right)y_k = \left\|\sum\limits_{j=1}^{d} A_{j, k}^*A_{j, k}\right\|y_k=\|\boldsymbol{A^*A}\|y_k.$$ Consider  $x=U^*(\sqrt{s_1}y_1,\dots, \sqrt{s_m}y_m, 0, \dots, 0)$. Then $\|x\|^2=1.$ Now, we have \begin{align*}\boldsymbol{A^*A} x & =\sum\limits_{j=1}^d A_j^*A_j x\\
& = U^*\left(\sum\limits_{j=1}^d (UA_j U^*)^*(UA_j U^*) Ux\right)\\
& = U^* \begin{bmatrix}\sum\limits_{j=1}^d A_{j,1}^*A_{j,1}& 0& \ldots&0\\ 0& \sum\limits_{j=1}^d A_{j,2}^*A_{j,2}& \ddots&0\\ \vdots&\vdots&\ddots&\vdots\\ 0&\ldots&0& \sum\limits_{j=1}^d A_{j,\ell}^*A_{j,\ell}\end{bmatrix} \begin{bmatrix} \sqrt{s_1} y_1\\ \sqrt{s_2} y_2\\\vdots\\\sqrt{s_m}y_m\\ 0\\\vdots\\0\end{bmatrix}\\
& = U^* \left(\sqrt{s_1}\|\boldsymbol{A^*A}\| y_1, \dots, \sqrt{s_m}\|\boldsymbol{A^*A}\| y_m, 0, \dots, 0\right).
\end{align*}
This gives \begin{equation}\label{799} \boldsymbol{A^*A} x = \|\boldsymbol{A}\|^2 x.
\end{equation}
Further, we have \begin{align*} \langle x| A_j x\rangle & = \langle Ux| UA_j U^* (Ux)\rangle\\
    & = \left\langle \left(\sqrt{s_1} y_1, \dots, \sqrt{s_m} y_m, 0, \dots, 0\right)| \left(\sqrt{s_1} A_{j, 1} y_1, \dots, \sqrt{s_m} A_{j, m} y_m, 0, \dots, 0\right)\right\rangle \\
    & = \sum\limits_{k=1}^m s_k \langle y_k | A_{j, k} y_k\rangle.
\end{align*} So by \eqref{589}, \eqref{eq10} and \eqref{799}, we get that $\boldsymbol{\lambda}\in\mathcal V(\boldsymbol{A})$. Thus \begin{equation*}\conv\left\{\mathcal V(A_{1,k},\dots, A_{d,k}): 1\leq k\leq \ell, \left\|\sum_{j=1}^d A_{j,k}^*A_{j,k}\right\|=\|\boldsymbol{A}\|^2\right\}\subseteq\mathcal V(\bA).\end{equation*} Hence the result.
\end{proof}


Now we show that \eqref{main2} holds for a doubly commuting tuple of matrices.
\begin{thm}\label{main thm}
For a doubly commuting tuple of matrices, we have
$$\dist  (\bA, \mathbb C^d \I)^2=\max_{\|x\|=1}\var_x (\bA).
$$
\end{thm}
\begin{proof}
By \eqref{main3}, we have $\max\limits_{\|x\|=1}\var\limits_x (\bA)\leq \dist  (\bA, \mathbb C^d \I)^2$. For the other side, we note that 
for every $\z\in \mathbb C^d$, $\|\bA^0-\z \I\|\geq \|\bA^0\|$. By Corollary \ref{orthogonality} and Lemma \ref{doubly}, there exists a unit vector $x$  such that \begin{equation*}\|\textstyle\bA^{0} x\|=\|\textstyle\bA^0\| \label{4}\end{equation*} and \begin{equation*}\langle x|\textstyle A^0_j x\rangle=0 \text{ for all }j=1,\ldots, d.\label{5}\end{equation*}
So \begin{eqnarray*}
\dist(\bA, \mathbb C^d \I)^2&=& \|\textstyle\bA^0 x\|^2\\
&=& \sum_{j=1}^d \|A^0_j x\|^2\\
&=&\sum_{j=1}^d \|A_j x-z^0_j x\|^2\\
&=&\sum_{j=1}^d \big(\|A_j x\|^2-\big|\langle x|A_j x \rangle\big|^2\big).
\end{eqnarray*}
Hence $\dist(\bA, \mathbb C^d \I) ^2 \leq \max\limits_{\|x\|=1}\var\limits_x (\bA) $.\end{proof}

Theorem 3.2 of \cite{bhatiasharma} follows as a special case of Theorem \ref{main thm}.

\section{Other tuples of operators}\label{section2}

First, we observe that for $d=1$, \eqref{main2} is true for any operator on any Hilbert space. To see this, in view of \eqref{main3}, only one side of the inequality is to be shown. Let $\lambda_0\in\bC$ be such that $\dist  (A, \mathbb C I)=\|A-\lambda_0 I\|$. Using Remark 3.1 of \cite{Bhatia}, there exists a sequence of unit vectors $x^{(n)}\in\mathcal H$ such that $\|(A-\lambda_0 I)x^{(n)}\|\rightarrow \dist  (A, \mathbb C I)$ and $\bra x^{(n)}|A x^{(n)}\ket\rightarrow \lambda_0$. So  
\begin{align*}
\dist(A, \mathbb C I)^2&= \lim_{n\rightarrow \infty}\|(A-\lambda_0 I) x^{(n)}\|^2\\
&=  \lim_{n\rightarrow \infty}\left(\|Ax^{(n)}\|^2-2\h\left(\lambda_0\langle A x^{(n)}|x^{(n)} \rangle\right) + |\lambda_0|^2\right)\\
&=\lim_{n\rightarrow \infty}\|Ax^{(n)}\|^2-|\lambda_0|^2\\
&=\lim_{n\rightarrow \infty} \var_{x^{(n)}} (A)\\
&\leq \sup_{\|x\|=1}\var_x (A).
\end{align*}

Equation \eqref{main2} may not hold for $d>1$. To show this, we give  examples for every such $d$. These are motivated by the example of Pauli spin matrices for the case $d=3$.

\textbf{Example 1}. Let  $\bA=(A_1, A_2, A_3)$, where $A_1=\begin{bmatrix}0&1\\ 1& 0 \end{bmatrix},A_2=\begin{bmatrix}0&-i\\ i& 0 \end{bmatrix},$ and $ A_3= \begin{bmatrix}1&0\\ 0& -1 \end{bmatrix}.$ Then for any $\z=(z_1, z_2, z_3)\in\bC^3$, 
\begin{align*}\|\bA-\z \I\|^2 &=\left\|\left(\begin{bmatrix}-z_1&1\\ 1& -z_1 \end{bmatrix},\begin{bmatrix}-z_2&-i\\ i& -z_2 \end{bmatrix} , \begin{bmatrix}1-z_3&0\\ 0& -1-z_3 \end{bmatrix}\right)\right\|^2\\
&=\left\|\begin{bmatrix}|z_1|^2+|z_2|^2+2+|1-z_3|^2&-(z_1+\overline{z_1})+i(z_2+\overline{z_2})\\ -(z_1+\overline{z_1})-i(z_2+\overline{z_2})& |z_1|^2+|z_2|^2+2+|1+z_3|^2\end{bmatrix}\right\|\\
&\geq \max\{|z_1|^2+|z_2|^2+2+|1-z_3|^2, |z_1|^2+|z_2|^2+2+|1+z_3|^2\}\\
&> 2.
\end{align*}
Now let $x=(x_1,x_2)\in\bC^2$. Then $$A_1x=(x_2,x_1), A_2x=(-ix_2, ix_1) \text{ and }A_3 x=(x_1, -x_2).$$ So $\|\bA x\|^2=3(|x_1|^2+|x_2|^2)$ and 
$$\sum\limits_{j=1}^3\big|\bra x|A_jx\ket\big|^2=4|x_1|^2|x_2|^2+(|x_1|^2-|x_2|^2)^2=(|x_1|^2+|x_2|^2)^2.$$ Hence for a unit vector $x\in\bC^2$, $\var\limits_{x} (\bA)= 2$. So $\dist(\bA, \bC^2\I)^2>\max\limits_{\|x\|=1}\var\limits_{x} (\bA)$.

Let $I$ be the $2\times 2$ identity matrix. For $d>3$, it is easy to see that $(A_1, A_2, A_3, I, \dots, I)$ works as a counterexample. 

For $d=2$, $\boldsymbol{A}=\left(\dfrac{1}{2}(A_1+iA_2), A_3\right)$ is a counterexample. We have $$\boldsymbol{A}= \left(\begin{bmatrix}0&1\\ 0& 0 \end{bmatrix}, \begin{bmatrix}1&0\\ 0& -1 \end{bmatrix}\right).$$ Then for any $\boldsymbol{z}=(z_1, z_2)\in\mathbb C^2$, 
\begin{align*}\|\bA-\z \I\|^2 &=\left\|\left(\begin{bmatrix}-z_1&1\\ 0& -z_1 \end{bmatrix}, \begin{bmatrix}1-z_2&0\\ 0& -1-z_2 \end{bmatrix}\right)\right\|^2\\
&=\left\|\begin{bmatrix}|z_1|^2+|1-z_2|^2&-\overline{z_1}\\ -z_1& |z_1|^2+1+|1+z_2|^2\end{bmatrix}\right\|\\
&\geq \max\{|1-z_2|^2, 1+ |1+z_2|^2\}\\
&\geq \max\{2-|1+z_2|^2, 1+|1+z_2|^2\}\\
&\geq 3/2.
\end{align*}
Now let $x=(x_1,x_2)\in\bC^2$. Then $$\begin{bmatrix}0&1\\ 0& 0 \end{bmatrix}x=(x_2,0), \begin{bmatrix}1&0\\ 0& -1 \end{bmatrix}x=(x_1, -x_2).$$ So $\|\bA x\|^2=|x_2|^2+|x_1|^2+|x_2|^2$ and 
$$\left|\left\bra x\bigg|\begin{bmatrix}0&1\\ 0& 0 \end{bmatrix}x\right\ket\right|^2+\left|\left\bra x\bigg| \begin{bmatrix}1&0\\ 0& -1 \end{bmatrix}x\right\ket\right|^2=|x_1|^2|x_2|^2+(|x_1|^2-|x_2|^2)^2.$$ Thus \begin{align*}\max_{\|x\|=1}\var\limits_{x} (\bA) &= \max_{\|x\|=1}\left(1+|x_2|^2-|x_1|^2|x_2|^2-(|x_1|^2-|x_2|^2)^2\right)\\
&=\max_{\|x\|=1}\left(1+|x_2|^2+3|x_1|^2|x_2|^2-(|x_1|^2+|x_2|^2)^2\right)\\
&= \max_{\|x\|=1}|x_2|^2(1+3|x_1|^2)\\
&= \max_{s\in [0,1]}(1-s)(1+3s)\\
& = 4/3\\
& < 3/2.\end{align*}  So $\dist(\bA, \bC^2\I)^2>\max\limits_{\|x\|=1}\var\limits_{x} (\bA).$

We also note that when $d=1$,
\begin{equation*}\label{new*} \dist({A}, \mathbb C {I}) =\sup\left\{\big|\bra {A}x| {y}\ket\big| : x,y\in\Hil, \|x\|=\|{y}\|=1, 
\bra x|y\ket =0\right\}.\end{equation*}
A proof of this can be found in \cite[Theorem 12.59]{Twopages} or \cite[Proposition 2.11]{LiRodman}.
  
\begin{thm}\label{00010}  The following are equivalent.

\begin{enumerate}[label=(\roman*)]

\item $\dist(\boldsymbol{A}, \bC^d\boldsymbol{I})^2 =  \sup\limits_{\substack{\|x\|=1}} \var\limits_x(\bA).$\label{two}
\item $\boldsymbol{0}\in W_0(\bA^0)$, that is, there is a sequence of unit vectors $(x^{(n)})$ such that $\lim\limits_{n\rightarrow\infty}\|\boldsymbol{A^0}x^{(n)}\|= \|\boldsymbol{A^0}\|$ and $\lim\limits_{n\rightarrow\infty}\bra x^{(n)}|A_j^0 x^{(n)}\ket = 0$ for all $1\leq j\leq d$. \label{one}
\item $\dist(\boldsymbol{A}, \bC^d\boldsymbol{I}) =\sup\left\{\big|\bra \boldsymbol{A}x| \boldsymbol{y}\ket\big| : x\in\Hil, \boldsymbol{y}\in\Hil^d, \|x\|=\|\boldsymbol{y}\|=1, \right.\\
\bra x|y_j\ket =0 \text{ for all } 1\leq j\leq d\big\}.$\label{three}
\end{enumerate}
\end{thm}
\begin{proof}  Suppose $\labelcref{two}$ holds. Since $\var\limits_x(\bA)=\var\limits_x(\bA^0)$, we have \begin{equation*}\label{109} \|\boldsymbol{A^0}\|^2 = \sup_{\|x\|=1}\var_x(\textstyle\bA^0).\end{equation*} So there is a sequence of unit vectors $(x^{(n)})$ such that $\lim\limits_{n\rightarrow \infty}\var\limits_{x^{(n)}}(\textstyle\bA^0)=\|\boldsymbol{A^0}\|^2$. Since $$\var\limits_{x^{(n)}}(\textstyle\bA^0)\leq \|\bA^0 x^{(n)}\|^2\leq \|\bA^0\|^2,$$ we get $\labelcref{one}$. The proof of $\labelcref{one}$ implies $\labelcref{two}$ is similar to the case $d=1$. 

Now let us assume $\labelcref{one}$. 
Then for each $1\leq j\leq d$, there is a unique $y_j^{(n)}\in(\bC x^{(n)})^{\perp}$ and a unique $\alpha_j^{(n)}\in\bC$ such that $$A^0_jx^{(n)} = \alpha_j^{(n)} x^{(n)}+y_j^{(n)}.$$ By the assumption,  $\lim\limits_{n\rightarrow\infty}\alpha_j^{(n)} = 0$ and $\lim\limits_{n\rightarrow\infty}\|(y_1^{(n)}, \dots, y_d^{(n)})\| =  \|\boldsymbol{A^0}\|$.
Let $\boldsymbol{y^{(n)}} = \dfrac{(y_1^{(n)}, \dots, y_d^{(n)})}{\|(y_1^{(n)}, \dots, y_d^{(n)})\|}$. Then
\begin{align*}\bra \boldsymbol{A}x^{(n)}| \boldsymbol{y^{(n)}}\ket & = \big\bra \boldsymbol{A^0}x^{(n)}\big| \boldsymbol{y^{(n)}}\big\ket \\
& = \dfrac{\|\boldsymbol{A^0}x^{(n)}\|^2 - \big\bra \boldsymbol{A^0}x^{(n)}\big|(\alpha_1^{(n)}x^{(n)}, \dots, \alpha_d^{(n)}x^{(n)})\big\ket}{\|(y_1^{(n)}, \dots, y_d^{(n)})\|}\\
&\rightarrow \|\boldsymbol{A^0}\| \text{ as } n\rightarrow\infty.\end{align*} This gives $$\dist(\boldsymbol{A}, \bC^d\boldsymbol{I}) \leq\sup\left\{\big|\bra \boldsymbol{A}x| \boldsymbol{y}\ket\big| : \|x\|=\|\boldsymbol{y}\|=1, \bra x|y_j\ket =0 \text{ for all } 1\leq j\leq d\right\}.$$ If $x\in\mathcal H$ and $\boldsymbol{y}\in\mathcal H^d$ are such that $\|x\|=\|\boldsymbol{y}\|=1$ and $\bra x|y_j\ket =0$ for all $ 1\leq j\leq d$, then for every $\boldsymbol{z}\in\mathbb C^d$, we have $$\big|\bra \boldsymbol{A}x| \boldsymbol{y}\ket\big| = \big|\bra (\boldsymbol{A-zI})x| \boldsymbol{y}\ket\big|\leq \|\boldsymbol{A-zI}\|.$$ So we also have $$\sup\left\{\big|\bra \boldsymbol{A}x| \boldsymbol{y}\ket\big| : \|x\|=\|\boldsymbol{y}\|=1, \bra x|y_j\ket =0 \text{ for all } 1\leq j\leq d\right\}\leq \dist(\boldsymbol{A}, \bC^d\boldsymbol{I}).$$ This gives $\labelcref{three}$.

Now assume $\labelcref{three}$ holds. We prove $\labelcref{one}$ along the lines of Remark 3.1 of \cite{Bhatia} where a similar result is shown for $d=1$. By our assumption, we get sequences of unit vectors $x^{(n)}\in\Hil$ and $\boldsymbol{y^{(n)}}\in\Hil^d$ such that $$\big|\bra \boldsymbol{A}^0 x^{(n)}|\boldsymbol{y^{(n)}}\ket\big|\rightarrow \|\boldsymbol{A}^0\|$$ and $$\bra x^{(n)}|y^{(n)}_j\ket =0\text{ for all } j=1,\dots, d.$$ Since $$\big|\bra \boldsymbol{A}^0 x^{(n)}|\boldsymbol{y^{(n)}}\ket\big|\leq \|\boldsymbol{A}^0 x^{(n)}\|\leq \|\boldsymbol{A}^0\|,$$ we get $$\lim\limits_{n\rightarrow \infty} \|\boldsymbol{A}^0 x^{(n)}\|=\|\boldsymbol{A}^0\|.$$ This gives $\boldsymbol{y^{(n)}}-\dfrac{\bA^0 x^{(n)}}{\|\bA^0 x^{(n)}\|}\rightarrow 0$. So $$\lim\limits_{n\rightarrow \infty}\bra x^{(n)}|A_j^0 x^{(n)}\ket =\|\boldsymbol{A}^0\|\lim\limits_{n\rightarrow \infty}\bra x^{(n)}|y^{(n)}_j\ket =0.$$ Thus we get $\labelcref{one}$.
\end{proof}

We have the following general result to Corollary \ref{orthogonality} for a tuple of operators $\bA$ for which $W_0(\bA)$ is convex.

\begin{prop}\label{orthogonality2}
Suppose $W_0(\bA)$ is convex. Then $$\|\bA-\z \I\|\geq \|\bA\|\text{ for all } \z\in \mathbb C^d$$ if and only if $\boldsymbol{0}\in W_0(\bA)$. \end{prop}
\begin{proof} The proof is similar to that of \cite[Theorem 2]{derivation} where this result is proved for $d=1$. We provide the details for the sake of completeness. Suppose there exists a sequence of unit vectors $x^{(n)}$ such that  $\lim\limits_{n\rightarrow \infty}\|\bA x^{(n)}\|=\|\bA\|$ and $\lim\limits_{n\rightarrow \infty}\langle x^{(n)}|A_j x^{(n)} \rangle= 0$ for all $j=1,\ldots,d$. Then for $\z\in\bC^d$, \begin{eqnarray*}\|(\bA-\z\I)x^{(n)}\|^2&=&\|\bA x^{(n)}\|^2-2\h\left(\sum\limits_{j=1}^d z_j\bra x^{(n)}|A_j x^{(n)}\ket\right)+\|\z\|^2\\
&\rightarrow& \|\bA\|^2+\|\z\|^2 \text{ as }n\rightarrow \infty.\end{eqnarray*} Thus $\|\bA-\z\I\|^2\geq \|\bA\|^2+\|\z\|^2$. In particular, $\|\bA-\z\I\| \geq\|\bA\|$.

Suppose $\|\bA-\z \I\|\geq \|\bA\|\text{ for all } \z\in \mathbb C^d$ and $\boldsymbol{0}\notin W_0(\bA)$. Since $W_0(\bA)$ is a closed and convex set, the Hahn-Banach separation theorem gives a $\tau>0$ and a unit vector $(w_1, w_2, \dots, w_d)\in\bC^d$ such that $$\h\sum\limits_{j=1}^dw_i\lambda_i\geq\tau>0 \text{ for all } (\lambda_1, \dots, \lambda_d)\in W_0(\bA).$$ Let $\mathcal S=\left\{x\in\Hil : \|x\|=1 \text{ and }\h\left(\sum\limits_{j=1}^dw_i\bra A_i x|x\ket\right)\leq\tau/2\right\}$. Let $\eta=\sup\{\|\bA x\|:x\in\mathcal S\}$. Clearly $\eta<\|\bA\|$. 
 Let $\mu=\min\{\tau/2, (\|\bA\|-\eta)/2\}$.  Let $\boldsymbol{w^0}=\mu(\overline{w_1}, \overline{w_2}, \dots, \overline{w_d})$. If $x\in\mathcal S$, then \begin{eqnarray*}\|(\bA-\boldsymbol{w^0}\I)x\|&\leq& \|\bA x\|+\mu\|(w_1, w_2, \dots, w_d)\|\\ &\leq& \eta+\mu\\ &<&\|\bA\|.\end{eqnarray*} Suppose $x$ is a unit vector such that $x\notin\mathcal S$. We write $\bA x$ as $$\bA x=\left((\alpha_1+i\beta_1)x, (\alpha_2+i\beta_2)x, \dots,(\alpha_d+i\beta_d)x\right)+\boldsymbol{y},$$ where $\bra\boldsymbol{y}|(x,x,\dots,x)\ket=0$ and $\alpha_j, \beta_j\in\mathbb R$ for all $1\leq j\leq d$. Then \begin{eqnarray*}\|(\bA-\boldsymbol{w^0}\I)x\|^2&=&\sum\limits_{j=1}^d(\alpha_i-\mu\h(w_i))^2+\sum\limits_{j=1}^d(\beta_i+\mu\hl(w_i))^2+\|\boldsymbol{y}\|^2\\
&=&\|\bA x\|^2+\mu^2-2\mu\sum\limits_{j=1}^d(\h(w_i)\alpha_i-\hl(w_i)\beta_i)\\
&=& \|\bA x\|^2+\left(\mu^2-2\mu\h\left(\sum\limits_{j=1}^dw_i(\alpha_i+i\beta_i)\right)\right)\\
&<&\|\bA\|^2-\mu^2.
\end{eqnarray*}
The last inequality follows because $\h\left(\sum\limits_{j=1}^dw_i(\alpha_i+i\beta_i)\right)>\tau/2\geq \mu$.
This gives $\|(\bA-\boldsymbol{w^0}\I)x\|<\|\bA\|$, which contradicts our assumption.\end{proof}

Combining Theorem \ref{00010} and Proposition \ref{orthogonality2}, we get the following. 

\begin{corollary}\label{corollary} Let $\bA$ be a tuple of operators on a Hilbert space $\Hil$. If $W_0(\bA^0)$ is convex, then \begin{equation*}
\sup_{\|x\|=1}\var_x (\bA)=\dist  (\bA, \mathbb C^d \I)^2. 
\end{equation*}
\end{corollary}

In particular, we get that \eqref{main2} holds for tuples of Toeplitz operators. Note that these tuples are non-commuting.

\begin{corollary}\label{main thm1}
For a tuple of Toeplitz operators $\bA$, we have
\begin{equation*}
\dist  (\bA, \mathbb C^d \I)^2=\sup_{\|x\|=1}\var_x (\bA).
\end{equation*}
\end{corollary}
\begin{proof} Note that $\bA^0$ is also a Toeplitz operator. It is also known that $W(\bA^0)$ is convex (by Theorem 2.6 of \cite{dash}), and thus so is $W_0(\bA^0)$. Using Corollary \ref{corollary}, we get the result.
\end{proof}

So far, convexity of $W_0(\bA^0)$ has been an important tool in obtaining \eqref{main2}. However, the below example shows that it is not necessary.

\textbf{Example 2}. Let  $A_1=\begin{bmatrix}1&0\\ 0& 0 \end{bmatrix}$ and $A_2=\begin{bmatrix}0&0\\ 1& 0 \end{bmatrix}.$ Let $\bA=(A_1,A_2)$. For any $\boldsymbol{z}=(z_1,z_2)\in\bC^2$,
\begin{align*}\|\boldsymbol{A-zI}\|^2 &=\left\|\left(\begin{bmatrix}1-z_1&0\\ 0& -z_1 \end{bmatrix},\begin{bmatrix}-z_2&0\\ 1& -z_2 \end{bmatrix} \right)\right\|^2 \\
&=\left\|\begin{bmatrix}|1-z_1|^2+|z_2|^2+1&-z_2\\ -\overline{z_2}& |z_1|^2+|z_2|^2\end{bmatrix}\right\| \\
&\geq \left\|\begin{bmatrix}|1-z_1|^2+|z_2|^2+1&-z_2\\ -\overline{z_2}& |z_1|^2+|z_2|^2 \end{bmatrix}\begin{bmatrix}1\\0 \end{bmatrix}\right\| \\
&= \left((|1-z_1|^2+|z_2|^2+1)^2+|z_2|^2\right)^{1/2} \\
&\geq 1.\end{align*}
Equality occurs if and only if $z_1=1$ and $z_2=0$. Thus $$\dist(\bA, \bC^2\I)^2=1=\|(A_1,A_2)-(I,0)\|^2.$$ So $\bA^0=\left(\begin{bmatrix}0&0\\ 0&-1 \end{bmatrix},\begin{bmatrix}0&0\\ 1&0 \end{bmatrix} \right)$. By definition, $$W_0(\boldsymbol{A}^0)=\mathcal V(\boldsymbol{A}^0)=\{(\langle x|A_1^0x\rangle,\langle x|A_2^0x\rangle) : \|x\|=1, (\boldsymbol{A}^0)^*\boldsymbol{A}^0 x = \|\boldsymbol{A^0}\|^2 x\}.$$ Since $(\bA^0)^*\bA^0 =I$, $$W_0(\boldsymbol{A}^0) = W(\boldsymbol{A}^0) = \left\{\left(-|z_2|^2, z_2\overline{z_1}\right): z_1, z_2\in\bC, |z_1|^2+|z_2|^2=1\right\}.$$ This set is not convex (see \cite[p. 138]{bonsall}).
Now let $x=(x_1, x_2)\in\bC^2$. Then
 $$\|\bA x\|^2=2|x_1|^2, \bra x|A_1 x\ket=|x_1|^2 \text{ and  }\bra x|A_2x\ket=x_2\overline{x_1}.$$ Hence $$\var\limits_{x} (\bA)=2|x_1|^2-|x_1|^4-|x_1|^2|x_2|^2. $$  
  So $$\max\limits_{\|x\|=1}\var\limits_{x} (\bA)=1=\dist(\bA, \bC^2\I)^2.$$

\section{remarks}\label{section3}

\textbf{Remark 1}. Theorem 1.3 of \cite{2021} gives that  $\|\bA-\z \I\|\geq \|\bA\|$ for all $\z\in \mathbb C^d$ if and only if there exists an at most countable set $\mathcal J$, a set of positive numbers $\{s_j: j\in\mathcal J\}$ and an  orthonormal set $\{x_j: j\in\mathcal J\}$ such that 
\begin{enumerate}
\item[(i)]  $\sum\limits_{j\in\mathcal J} s_j =1$ ,
\item[(ii)]$\bA^* \bA x_j = \|\bA\|^2x_j$  for each $j\in \mathcal J$,
\item[(iii)] $\sum\limits_{j\in\mathcal J} s_j\bra x_j |A_i x_j\ket = 0$ for each $1\leq i\leq d$. 
\end{enumerate}
Another proof of Theorem \ref{orthogonality1} now follows using the following fact. 
Let $X$ be a convex subset of $\mathbb C^d$. Let $(\boldsymbol{z_n})$ be a  sequence of elements in $X$ and let $(s_n)$ be a sequence of non-negative numbers such that $\sum\limits_{n=1}^{\infty}s_n =1$. If $\sum\limits_{n=1}^{\infty}s_n \boldsymbol{z_n}$ exists, then it lies in $X$. 

\textbf{Remark 2}.  Let $\bA$ be a tuple of commuting matrices. We note that when each $A_j$ is a $2\times 2$ or $3\times 3$ matrix, $W(\bA)$ is convex  (see Proposition 2.3 and Theorem 3.4, respectively, of \cite{newpaper}). Thus in these cases, \eqref{main2} holds. Any collection of mutually commuting matrices have a common eigenvector. Using this fact and following the idea of the proof of \cite[Lemma 5]{connection}, we obtain
\begin{align*} \big\{\bra \boldsymbol{A}x| \boldsymbol{y}\ket &: x\in\Hil, \boldsymbol{y}\in\Hil^d, \|x\|=\|\boldsymbol{y}\|=1, \bra x|y_j\ket =0 \text{ for all } 1\leq j\leq d\big\}\\
&=\big\{z\in\bC: |z|^2\leq\var_x (\bA)\text{ for some } x\in\Hil, \|x\|=1\big\}.
\end{align*}
For a tuple of operators, we also have equivalence of $\labelcref{two}$ and $\labelcref{three}$ in Theorem \ref{00010}. This raises the curiosity if  \eqref{main2} holds for every tuple of commuting operators. 

\textbf{Remark 3}.  The condition $\|\bA-\z \I\|\geq \|\bA\|\text{ for all } \z\in \mathbb C^d$ is same as saying $\boldsymbol{0}$ is a best approximation of $\bA$ to the subspace $\bC^d\I$ of $\mathscr B(\Hil,\Hil^d)$ (see \cite{Singer}). In other words, we say $\bA$ is \emph{Birkhoff-James orthogonal} to $\bC^d\I$ (see \cite{James2}). Proposition \ref{orthogonality2} is a characterization for $\bA$ to be Birkhoff-James orthogonal to the subspace $\bC^d\I$. Some characterizations of Birkhoff-James orthogonality of an element to a subspace can be found in  \cite{2014, 2017, 2019, GroverSingla, Keckic3, Keckic2, arxiv, 2021, Zietak2}.

\textbf{Remark 4}. From the proof of Proposition 8, it follows that if $W_0(\bA)$ is convex, then the following are equivalent.
\begin{enumerate}
\item[(i)] $\|\bA-\z \I\|\geq \|\bA\|\text{ for all } \z\in \mathbb C^d$
\item[(ii)] $\boldsymbol{0}\in W_0(\bA)$
\item[(iii)] $\|\bA-\z \I\|^2\geq \|\bA\|^2+\|\z\|^2\text{ for all } \z\in \mathbb C^d.$
\end{enumerate}

\textbf{Acknowledgements.} The first named author's research is supported by an Early Career Research Award (ECR/2018/001784) of SERB, India. The authors would like to thank Rajendra Bhatia for suggesting the possibility of exploring this problem for tuples. The authors are extremely thankful to the referee for her/his comments.

\end{document}